\documentclass[11pt]{article}
\usepackage[colorlinks=true, urlcolor=blue, linkcolor=red]{hyperref}
\usepackage[utf8]{inputenc}
\usepackage{mathtools}
\usepackage{latexsym}
\usepackage{amsmath}
\usepackage{cleveref}
\usepackage{nicematrix}

\usepackage{amssymb}
\usepackage{amsthm}
\usepackage{epsfig}
\usepackage{color}
\usepackage{hyperref}
\usepackage{multicol}
\usepackage{pstricks}
\usepackage{pst-node}
\usepackage{pst-tree}
\usepackage{pst-coil}
\usepackage{multirow}
\usepackage{graphics}
\usepackage{tikz}
\usetikzlibrary{automata, positioning, arrows}

\makeatletter
\DeclareRobustCommand\bigop[1]{%
  \mathop{\vphantom{\sum}\mathpalette\bigop@{#1}}\slimits@
}
\newcommand{\bigop@}[2]{%
  \vcenter{%
    \sbox\z@{$#1\sum$}%
    \hbox{\resizebox{\ifx#1\displaystyle.9\fi\dimexpr\ht\z@+\dp\z@}{!}{$\m@th#2$}}%
  }%
}
\makeatother

\theoremstyle{definition}
\newtheorem{theorem}{Theorem}

\newtheorem{proposition}[theorem]{Proposition}
\newtheorem{definition}[theorem]{Definition}

\newtheorem{claim}[theorem]{Claim}

\newtheorem{example}[theorem]{Example}

\advance \topmargin by -\headheight
\advance \topmargin by -\headsep
\evensidemargin \oddsidemargin
\begin{document}

\title{Transcendence of Hecke-Mahler Series}

\author{
  {\sc Florian~Luca}\\
  {Mathematics Division, Stellenbosch University}\\
{Stellenbosch, South Africa}\\
{fluca@sun.ac.za}
\and
{\sc Jo\"el Ouaknine}\\
{Max Planck Institute for Software Systems}\\
{Saarland Informatics Campus, Saarbr\"ucken, Germany  }\\
{joel@mpi-sws.org}
\and
{\sc James Worrell}\\
{Department of Computer Science}\\
{University of Oxford, Oxford OX1 3QD, UK}\\
{jbw@cs.ox.ac.uk}}

\date{}
\maketitle

\begin{abstract}
  We prove transcendence of the Hecke-Mahler series
  $\sum_{n=0}^\infty f(\lfloor n\theta+\alpha \rfloor) \beta^{-n}$,
  where $f(x) \in \mathbb{Z}[x]$ is a non-constant polynomial,
  $\alpha$ is a real number, $\theta$ is
  an irrational real number, and $\beta$ is an algebraic number such
  that $|\beta|>1$.
  \end{abstract}

\section{Introduction}

In this paper we prove transcendence of the 
Hecke-Mahler series
\begin{gather}\sum_{n=0}^\infty f(\lfloor n\theta+\alpha \rfloor) 
  \beta^{-n} \, ,
\label{eq:HM}
\end{gather}
where $f(x) \in \mathbb{Z}[x]$ is a non-constant polynomial, $\alpha$
is a real number, $\theta$ is an irrational real number, and $\beta$
is an algebraic number such that $|\beta|>1$.  Transcendence
of~\eqref{eq:HM} in the case $\alpha=0$ was proved by Loxton and Van
der Poorten~\cite[Theorem 7]{LP}.  Transcendence with $\alpha$
unrestricted and  $f(x)=x$  was recently shown by Bugeaud and
Laurent~\cite{BL} and by the present authors~\cite{LOW}.
The papers~\cite{BL,LP} use the Mahler method to prove transcendence,
while~\cite{LOW} relies on the  $p$-adic Subspace  Theorem.
Linear independence over $\overline{\mathbb Q}$ of
expressions of the form~\eqref{eq:HM}, again with $f(x)=x$, is
studied in~\cite{BKLN21,LOW}.  Meanwhile, Masser~\cite{Masser1999}
proves algebraic independence results for Hecke-Mahler series of the
form $\sum_{n=0}^\infty \lfloor n\theta \rfloor \beta^{-n}$ in case
$\theta$ is a quadratic irrational number.

To prove the transcendence of~\eqref{eq:HM} for arbitrary $f$, we introduce a new
combinatorial condition on a sequence of numbers
$\boldsymbol u=\langle u_n\rangle_{n=0}^\infty$ that, via the $p$-adic
Subspace
Theorem, entails transcendence of the sum
$\sum_{n=0}^\infty u_n \beta^{-n}$ for algebraic $\beta$ with
$|\beta|>1$.  This condition is a development of those presented
in~\cite{KebisLOS024,LOW} (which are in turn inspired by~\cite{AB1,FM}),
but with significant new elements.
Intuitively, the condition can be considered as a \emph{linear recurrence
measure} for $\boldsymbol u$, measuring the extent to which
$\boldsymbol u$
satisfies a linear recurrence of a  prescribed form.
\section{Preliminaries}
\label{sec:prelim}

Let $K$ be a number field of degree $d$ over $\mathbb Q$ and let $M(K)$ be the set of
\emph{places} of $K$.  We divide $M(K)$ into the collection of
\emph{Archimedean places}, which are determined either by an embedding of
$K$ in $\mathbb{R}$ or a complex-conjugate pair of embeddings of $K$
in $\mathbb{C}$, and the set of \emph{non-Archimedean places}, which are
determined by prime ideals in the ring $\mathcal{O}_K$ of integers of $K$.

For $a \in K$ and $v \in M(K)$, define the absolute value $|a|_v$ as
follows: $|a|_v := |\sigma(a)|^{1/d}$ if $v$ corresponds to a real
embedding $\sigma:K\rightarrow \mathbb{R}$;
$|a|_v := |\sigma(a)|^{2/d}$ if $v$ corresponds to a complex-conjugate
pair of embeddings
$\sigma,\overline{\sigma}:K \rightarrow \mathbb{C}$; and
$|a|_v := N(\mathfrak{p})^{-\mathrm{ord}_{\mathfrak{p}}(a)/d}$ if $v$
corresponds to a prime ideal $\mathfrak{p}$ in $\mathcal{O}$ and
$\mathrm{ord}_{\mathfrak{p}}(a)$ is the order to which $\mathfrak{p}$
divides the ideal $a\mathcal{O}$.  With the above definitions we have
the \emph{product formula}: $\prod_{v \in M(K)} |a|_v = 1$ for all
$a \in K^\times$.  Given a set of places $S\subseteq M(K)$, the ring
$\mathcal{O}_S$ of \emph{$S$-integers} is the subring comprising all
$a \in K$ such $|a|_v \leq 1$ for all non-Archimedean places
$v\not\in S$.

For $m\geq 1$ the \emph{absolute Weil height} of the projective point
$\boldsymbol{a}=[a_0 : a_1 : \cdots : a_m] \in \mathbb{P}^m(K)$ is 
\[ H(\boldsymbol{a}):=\prod_{v \in M(K)}\max(|a_0|_v,\ldots,|a_m|_v)
  \, .\] This definition is independent of the choice of the field $K$
containing $a_0,\ldots,a_m$.  We define the height $H(a)$ of $a \in K$
to be the height $H([1 : a])$ of the corresponding point in
$\mathbb{P}^1(K)$.  For a non-zero Laurent polynomial
$f = x^n \sum_{i=0}^m a_i x^i \in K[x,x^{-1}]$, where $m\geq 1$ and
$n\in \mathbb Z$, following~\cite{LEN97} we define its height $H(f)$
to be the height $H([a_0 : \cdots : a_m])$ of the vector of
coefficients.

The following special case\footnote{We formulate the special
  case of the Subspace Theorem in which all
  but one of the linear forms are coordinate variables.} of the $p$-adic
Subspace Theorem of Schlickewei~\cite{Schlickewei76} is one of the
main ingredients of our approach.
\begin{theorem}
  Let $S \subseteq M(K)$ be a finite set of places of $K$ that
  contains all Archimedean places.  Let $v_0 \in S$ be a distinguished
  place and choose a continuation of $|\cdot |_{v_0}$ to
  $\overline{\mathbb Q}$, also denoted $|\cdot |_{v_0}$.  Given
  $m\geq 2$, let $L(x_1,\ldots,x_{m})$ be a linear form with algebraic
  coefficients and let $i_0 \in \{1,\ldots,m\}$ be a distinguished
  index such that $x_{i_0}$ has non-zero coefficient in $L$.  Then for
  any $\varepsilon>0$ the set of solutions
  $\boldsymbol{a}=(a_1,\ldots,a_m) \in (\mathcal{O}_S)^m$ of the
  inequality
  \[ |L(\boldsymbol{a})|_{v_0}\cdot \Bigg( \prod_{\substack{(i,v)\in
        \{1,\ldots,m\}\times S\\(i,v) \neq (i_0,v_0)}}|a_i|_v \Bigg)
    \leq H(\boldsymbol{a})^{-\varepsilon} \] is contained in a finite union
  of proper linear subspaces of $K^m$.
\label{thm:SUBSPACE}
\end{theorem}

We will also need the following additional proposition about roots of
univariate polynomials.
\begin{proposition}{\cite[Proposition 2.3]{LEN97}}
  Let $f \in K[x,x^{-1}]$ be a Laurent polynomial with at most 
  $k+1$ terms.  Assume that $f$ can be written as the sum of two 
  polynomials $g$ and $h$, where every monomial of $g$ has degree at 
  most $d_0$ and every monomial of $h$ has degree at least $d_1$. 
  Let $\beta$ be a root of $f$ that is not a root of unity.  If 
  $d_1-d_0> \frac{\log (k \, H(f))}{\log H(\beta) }$ then $\beta$ is a 
  common root of $g$ and $h$. 
\label{prop:gap}
  \end{proposition}

  \section{A Transcendence Condition}
\label{sec:stuttering}
In this section we introduce a condition on
a sequence~$\boldsymbol u$ that implies that the number
$\sum_{m=0}^\infty u_m\beta^{-m}$ is transcendental.  Intuitively the
condition says that $\boldsymbol u$ almost satisfies a linear
recurrence.
\begin{example}
  Let the sequence $\boldsymbol{u}=\langle u_m\rangle_{m=0}^\infty$ be
  given by $u_m:=\left\lfloor \frac{m+1}{2+\sqrt{2}} \right\rfloor$.
  Consider also the sequence $\langle r_n \rangle_{n=0}^\infty =
  \langle 1, 3, 7, 17, 41, 99, 239, \ldots\rangle$ of numerators of the
  convergents of the continued fraction expansion of $\sqrt{2}$.  For
  all $n\in \mathbb{N}$ define   the
  sequence $\boldsymbol w_n=\langle w_{n,m}\rangle_{m=0}^\infty$ by
  $w_{n,m}:= u_{m+2r_n}-2u_{m+r_n}-u_m$. Then  $\boldsymbol w_n$
  becomes increasing sparse for successively larger $n$.  For example,
for $n=2$ it holds that $w_{n,m}$ is zero for 
$m \in \{0,\ldots,70\} \setminus \{9,16,26,33,50,57,67\}$,
for $n=3$ we have that $w_{n,m}$ is zero for 
  $m \in \{0,\ldots,70\}\setminus \{23,40,64\}$, while for $n=4$ we
  have that $w_{n,m}$ is zero for $m \in \{0,\ldots,70\} \setminus \{57\}$.
  \end{example}

  The following definition aims to capture the above behaviour.  Here,
  given $f,g: S \rightarrow \mathbb{R}_{\geq 0}$, we write $f \ll g$
  if there exists a constant $c>0 $ such that
  $f(a) \leq c g(a )$ for all $a \in S$.  We also write $f\lll g$ if for all $c>0$ the inequality $f(a) \leq cg(a)$ holds for
  all but finitely many $a \in S$.

      \begin{definition}
        An integer sequence
        $\boldsymbol{u}=\langle u_m\rangle_{m=0}^\infty$ satisfies
        Condition (*) if there exist $\sigma \geq 0$,
        $b_0,\ldots,b_\sigma \in\mathbb Z$, and an unbounded integer
        sequence $\langle r_n \rangle_{n=0}^\infty$ such that,
defining for each $n\in \mathbb N$
        the sequence
        $\langle w_{n,m}\rangle_{m=0}^{\infty}$ by
        $w_{n,m} := \sum_{i=0}^\sigma b_i u_{m+ir_n}$, the following
        two properties are satisfied:
        \begin{enumerate}
        \item \emph{Expanding Gaps}:
          for all $n \in \mathbb N$, the set $\Delta_n:=\{ m : w_{n,m}\neq 0\}$ is 
        infinite and the minimum distance $\mu_n$ between any two
        elements of $\Delta_n$ satisfies $\mu_n \gg r_n$.
      \item \emph{Polynomial Variation}:
there exists $c_0\geq 0$ such that for all $n\in\mathbb N$ and 
        all   $m'>m$ in $\Delta_n$,  $|w_{n,m'}| \ll
         (m'-m)^{c_0}+|w_{n,m}|$. 
        \end{enumerate}
  \label{def:stuttering}
\end{definition}
\begin{theorem}
  Let the integer sequence $\boldsymbol{u}=\langle
  u_m\rangle_{m=0}^\infty$ satisfy Condition (*) and be such that $|u_m|\ll m^{c_1}$ for some
  $c_1\geq 0$. 
  Then for any algebraic number $\beta$ such that
  $|\beta|>1$, the sum $\alpha:=\sum_{m=0}^\infty \frac{u_m}{\beta^m}$
  is transcendental.
\label{thm:main}
\end{theorem}
\begin{proof}
  Suppose that $\alpha$ is algebraic.  We will use the Subspace
  Theorem to obtain a contradiction.  Let $S$ comprise all the  Archimedean
  places of $\mathbb Q(\beta)$ and all non-Archimedean places corresponding to
  prime ideals $\mathfrak{p}$ of $\mathcal O_{\mathbb Q(\beta)}$
  such that $\mathrm{ord}_{\mathfrak{p}}(\beta)\neq 0$.  Let $v_0\in S$ be the place
  corresponding to the inclusion of $\mathbb Q(\beta)$ in
  $\mathbb{C}$.  Recall that
  $|\cdot|_{v_0}=|\cdot|^{1/\mathrm{deg}(\beta)}$, where $|\cdot|$ denotes the
  usual absolute value on $\mathbb{C}$.  Let $\kappa \geq 2$ be an
  upper bound of $|\beta|_v$ for all $v \in S$.

  By the assumption that $\boldsymbol u$ satisfies Condition (*), there
  is an integer sequence $\langle r_n\rangle_{n=0}^\infty$ and
  $b_0,\ldots,b_\sigma \in \mathbb Z$ such that the family of
  sequences $\langle w_{n,m}\rangle_{m=0}^{\infty}$ defined by
  $w_{n,m} := \sum_{i=0}^{\sigma} b_i u_{m+ir_n}$ satisfies the
Expanding-Gaps and Polynomial-Variation conditions from Definition~\ref{def:stuttering}.
  
  Define $\rho:=2\sigma |S| \mathrm{deg}(\beta) \frac{\log \kappa}{\log
    |\beta|}$.  For each $n\in \mathbb N$, let
  $0 \leq m_{n,1}< m_{n,2} < \cdots$ be an increasing enumeration of
  $\{m:w_{n,m}\neq 0\}$.  Since $\mu_n \gg r_n$, we may define
  $\delta\geq 1$ to be least such that $m_{n,\delta} > \rho r_n$ for
  infinitely many
  $n\in \mathbb N$.  We also define $s_n := m_{n,\delta+1}-1$ for all
  $n\in\mathbb N$ and note that $s_n-m_{n,\delta} \gg
  r_n$.
  
    For $n\in \mathbb{N}$, define
    $\boldsymbol{a}_n:=(a_{n,0},\ldots,a_{n,\sigma+\delta+1}) \in
    (\mathcal{O}_S)^{\sigma+\delta+2}$ by
   \begin{equation}
\begin{array}{rcll}
         a_{n,i} & \, := \, & \beta^{ir_n} \quad\text{for $i\in
  \{0,\ldots,\sigma\}$} \qquad\qquad\quad &
         a_{n,\sigma+1} \, := \,  \sum_{i=0}^\sigma \sum_{m=0}^{ir_n-1} 
                     b_i u_m\beta^{ir_n-m} \\[4pt]
      a_{n,\sigma+1+j} & \, := \, & w_{n,m_{n,j}}\, \beta^{-m_{n,j}} \quad\text{for
                       $j\in \{1,\ldots,\delta \}$}\, . &
\end{array}
                            \label{eq:a-n}
                   \end{equation}
and consider the linear form 
    \[ L(x_0,\ldots,x_{\sigma+\delta+1}):= \alpha \sum_{i=0}^\sigma b_i x_i -
      \sum_{i=\sigma+1}^{\sigma+\delta+1} x_i \, . \]
    Then for all $n\in\mathbb N$ we have
    \begin{gather}
      L(\boldsymbol{a}_n)  = \sum_{i=0}^{\sigma} \sum_{m=0}^{\infty}u_mb_i\beta^{ir_n-m} -
              \sum_{i=0}^{\sigma} \sum_{m=0}^{ir_n-1} b_i u_m\beta^{ir_n-m}- 
               \sum_{i=1}^\sigma  w_{n,m_{n,i}}\, \beta^{-m_{n,i}}
=                         \sum_{m=s_n+1}^\infty  w_{n,m}\beta^{-m}
                  \label{eq:lin-form}
\end{gather}
By the assumption that  $|u_m| \ll m^{c_1}$, there exists $c_2 \geq 0$ such that $|w_{n,m}| \ll
(m+r_n)^{c_2}$.
By~\eqref{eq:lin-form},
  \begin{equation}
  \textstyle  |L(\boldsymbol{a}_n)| =  \left | \sum_{m=s_n+1}^\infty w_{n,m}\beta^{-m} \right|
   \ll  \sum_{m=s_n+1}^\infty 
           (m+r_n)^{c_2} |\beta|^{-m} 
   \ll   s_n^{c_2}\,  |\beta|^{-s_n} \, .
\label{eq:STRICT}
\end{equation}

For $v \in S$, recalling that $|\beta|_v \leq \kappa$, there is a
constant $c_3$ such that 
\begin{gather}
  |a_{n,\sigma+1}|_v  \ll  r_n^{c_3} \kappa^{\sigma r_n} \,  .
\label{eq:BOUND1}
\end{gather}
By the product formula we have $\prod_{v\in S}|a_{n,i}|_v=1$ for $i
\in \{0,\ldots,\sigma\}$ and
\begin{gather}
  \prod_{j=1}^{\delta} \prod_{v \in S}  |a_{n,\sigma+1+j}|_v
  = \prod_{j=1}^\delta |w_{n,m_{n,j}}| \ll s_n^{c_2\delta} \,  .
\label{eq:BOUND2}
\end{gather}

The bounds~\eqref{eq:STRICT},~\eqref{eq:BOUND1},
and~\eqref{eq:BOUND2}, imply that there is a constant $c_4$ such that for
all $n$,
\begin{equation}
 \begin{array}{rcl}
   |L(\boldsymbol{a}_n)|_{v_0} \cdot
     \prod_{\substack{(i,v) \in \{0,\ldots,\sigma+\delta+1\}\times
   S\\(i,v) \neq (\sigma+1,v_0)}} |a_{i,n}|_v &\,\leq \, & 
 \kappa^{\sigma r_n|S|} s_n^{c_4} \,
 |\beta|^{-s_n/\mathrm{deg}(\beta)} \\
 & \leq & s_n^{c_4} |\beta|^{-s_n/2\mathrm{deg}(\beta)} \, ,
 \label{eq:INEQ3}
 \end{array}
 \end{equation}
where the second inequality follows from the fact that,
since $s_n\geq \rho r_n$ and 
$\rho=2\sigma |S| \mathrm{deg}(\beta) \frac{\log \kappa}{\log |\beta|}$,
we have $\kappa^{\sigma r_n|S|} = |\beta|^{\rho r_n / 2 \mathrm{deg}(\beta)}
\leq |\beta|^{s_n/2\mathrm{deg}(\beta)}$.
On the other hand, there
exists a constant $c_5>0$ such that the height of~$\boldsymbol a_n$
satisfies $H(\boldsymbol{a}_n) \ll |\beta|^{c_5 s_n}$.
Thus there exists $\varepsilon>0$ such that the right-hand side
of~\eqref{eq:INEQ3} is at most $H(\boldsymbol{a}_n)^{-\varepsilon}$
for $n$ sufficiently large.  We can therefore apply Theorem~\ref{thm:SUBSPACE}
to obtain a non-zero linear form $F(x_0,\ldots,x_{\sigma+\delta+1})$,
with coefficients in $\overline{\mathbb{Q}}$, such that
$F(\boldsymbol{a}_n)=0$ for infinitely many~$n$.

By Claim~\ref{claim:C}, below, the support of $F$ contains
variable $x_{\sigma+1}$ but omits variable
$x_{\sigma+\delta+1}$.
Thus, by subtracting a suitable
multiple of $F$ from $L$ we obtain a linear form $L'$ whose support
includes $x_{\sigma+\delta+1}$ but not $x_{\sigma+1}$ and 
such that $L'(\boldsymbol a_n)=L(\boldsymbol a_n)$ for
infinitely many $n$.

By Claim~\ref{claim:C}(i) we have 
$|L'(\boldsymbol a_n)| \gg |a_{n,\sigma+\delta+1}|= \left|
  w_{n,m_{n,\delta}} \beta^{-m_{n,\delta}} \right|$.
Expanding $L'(\boldsymbol a_n)=L(\boldsymbol a_n)$ as
in~\eqref{eq:lin-form}, for infinitely many $n$ we have
\begin{gather}
\left| w_{n,m_{n,\delta}} \beta^{-m_{n,\delta}} \right| \ll
\left|L'(\boldsymbol a_n)\right|  \, = \,\bigg| \sum_{j=\delta+1}^\infty w_{n,m_{n,j}}
                                            \beta^{-m_{n,j}} \bigg|   \, . 
\label{eq:GAT}
\end{gather}
The Polynomial-Variation Condition gives $|w_{n,m_{n,j}}|\ll
(m_{n,j}-m_{n,\delta})^{c_0}+|w_{n,m_{n,\delta}}|$ for all $j\geq
\delta+1$.  
Applying this upper bound to the right-hand sum in~\eqref{eq:GAT} and
dividing through by $|w_{n,m_{n,\delta}}|$, we obtain
\[  \left| \beta^{-m_{n,\delta}} \right| \ll 
     \left| (m_{n,\delta+1}- m_{n,\delta})^{c_0}  \beta^{-m_{n,\delta+1}}
     \right| \, . \]
But this  contradicts the fact that
   $m_{n,\delta+1}-m_{n,\delta}\rightarrow \infty$ as $n\rightarrow \infty$.
\end{proof}

\begin{claim}
      Let  $F(x_0,\ldots,x_{\sigma+\delta+1}) = \sum_{i=0}^{\sigma+\delta+1} \alpha_i
      x_i$ be  a linear form with coefficients in $\overline{\mathbb
        Q}$.  We have (i)~if $\alpha_{\sigma+1}=0$ then $|F(\boldsymbol
      a_n)| \gg |a_{n,\sigma+\delta+1}|$; (ii)~if
      $F(\boldsymbol{a}_n)=0$ for infinitely many $n$, then
      $\alpha_{\sigma+\sigma+1}=0$.
\label{claim:C}
\end{claim}
\begin{proof}
  Let $i,j\in \{1,\ldots,\delta\}$ be such that  $i<j$.  By the
  Polynomial-Variation Condition, for all $n$, 
\[ \left|\frac{a_{n,\sigma+1+j}}{a_{n,\sigma+1+i}}\right|
=\left| \frac{w_{n,m_{n,j}}}{w_{n,m_{n,i}}}\right| \left|\beta\right|^{-(m_{n,j}-m_{n,i})}
    \leq
  ((m_{n,j}-m_{n,i})^{c_0}+1) |\beta|^{-(m_{n,j}-m_{n,i})} \, .
\]
But as $n\rightarrow \infty$ we have
$m_{n,j}-m_{n,i}\rightarrow \infty$ and hence
$\left|\frac{a_{n,\sigma+1+j}}{a_{n,\sigma+1+i}}\right| \rightarrow 0$.
   Recalling that $a_{n,i}=\beta^{ir_n}$ for $i \in
   \{0,\ldots,\sigma\}$,  we can say more generally that 
  $|a_{n,i}| \ggg |a_{n,j}|$ for 
  $i,j \in \{0,\ldots,\sigma+\delta+1\} \setminus \{\sigma+1\}$
  with $i>j$.  
  Item~(i) immediately follows.
  
We next show Item~(ii).
  For all $n\in \mathbb{N}$, by~\eqref{eq:a-n} we have
  $F(\boldsymbol{a}_n) = P_n(\beta)$ for the polynomial
  $P_n(x) \in \overline{\mathbb{Q}}[x,x^{-1}]$ defined by
  \[ P_n(x) := \sum_{i=0}^\sigma \alpha_i x^{i r_n} + 
\alpha_{\sigma+1}\sum_{i=0}^{\sigma} \sum_{m=0}^{ir_n-1} b_i u_m x^{ir_n-m} 
+\sum_{j=1}^\delta \alpha_{\sigma+1+j} w_{n,m_{n,j}} x^{-m_{n,j}} \, .
\]
Suppose that
$P_n(\beta)=0$ for infinitely many $n$.   We show that
$\alpha_{\sigma+\delta+1}= 0$.  Indeed,
by the Expanding Gaps Condition we have that 
$m_{n,\delta-1}-m_{n,\delta} \gg r_n$.
Now $P_n(x)$ has at most $\sigma r_n+\delta$ monomials and,
since
since $m_{n,\delta-1}\leq \rho r_n$, $P_n(x)$
has height bounded polynomially in $m_{n,\delta-1}-m_{n,\delta}$.
Hence, by Proposition~\ref{prop:gap} we have that 
that for $n$ sufficiently large, if $P_n(\beta)=0$ then 
  $\alpha_{\sigma+\sigma+1}=0$.
\end{proof}

  \section{Hecke-Mahler Series}

  Write $I:=[0,1)$ for the unit interval.  Every $r \in \mathbb{R}$
  can be written uniquely in the form
  $r = \lfloor r \rfloor + \{ r\}$, where
  $\lfloor r \rfloor \in \mathbb Z$ is the \emph{integer part} of $r$
  and $\{ r\} \in I$ is the \emph{fractional part} of $r$.  Write also
  $\|r\|$ for the distance of $r$ to the nearest integer.  Let
  $0<\theta<1$ be an irrational number and define the \emph{rotation
    map} $R = R_\theta :I \rightarrow I$ by $R(r)= \{r+\theta \}$.
    
      Write $[a_0,a_1,a_2,a_3,\ldots]$ for the 
simple continued-fraction expansion of $\theta$.
Given $n\in \mathbb{N}$, we write
$p_n/q_n:=[a_0,a_1,\ldots,a_n]$ for the $n$-th convergent.  It
is well known that for all $n\in\mathbb{N}$ we have
\begin{gather}
\frac{1}{(a_{n+1}+2)q_n} < |q_n\theta - p_n|<\frac{1}{a_{n+1}q_n} \, . 
\label{eq:FRAC}
\end{gather}
We moreover have the \emph{law of best approximation}:
$q\in\mathbb{N}$ occurs as one of the $q_n$ just in case
$\|q\theta\|<\|q'\theta\|$ for all $q'$ with $0<q'<q$.

\begin{theorem}
  Let $\theta,\alpha \in (0,1)$ with $\theta$ irrational.  Given a
  non-constant polynomial $f(x)\in \mathbb{Z}[x]$, the series
  $\boldsymbol u = \langle u_n\rangle_{n=0}^\infty$ given by
  $u_n := f(\lfloor n \theta+\alpha \rfloor) $ satisfies Condition (*).
        \label{thm:main1}
  \end{theorem}
  \begin{proof}
    Referring to Definition~\ref{def:stuttering}, we define the
    sequence $\langle r_n\rangle_{n=0}^\infty$ via the
    continued-fraction expansion $[a_0,a_1,a_2,a_3,\ldots]$
    of~$\theta$.  If the expansion is unbounded, then choose
    $\ell_1 < \ell_2< \cdots$, either all odd or all even, such that
    $a_{\ell_n+1} \geq a_m$ for all $n\in\mathbb N$ and all
    $m \leq \ell_n$.  If the expansion is bounded, then choose
    $\ell_1 < \ell_2< \cdots$ to be the even natural numbers.  In
    either case, there exists a constant $\varepsilon>0$ such that for
    all $n\in\mathbb N$
    and $m\leq \ell_n$, we have
    $\frac{a_{{\ell_n}+1}}{a_{m+1}+2} \geq \varepsilon$.  Now define
    $r_n$ to be the denominator $q_{\ell_n}$ of the $\ell_n$-th
    convergent.  Since the $r_n$ all have the same parity, we have
    either that $\|r_n\theta\| = \{ r_n\theta\}$ for all $n$ or
    $\|r_n\theta\| = 1-\{ r_n\theta\}$ for all $n$.  We assume the
    former case; the reasoning in the latter case requires minor
    modifications.

    Let $\sigma:=\mathrm{deg}(f)+1$ and define the sequence
    $\boldsymbol w_n = \langle w_{n,m} \rangle_{m=0}^\infty$ by
    $w_{n,m}:=\sum_{k=0}^\sigma b_k u_{m+kr_n}$ where
    $b_k := (-1)^k \binom{\sigma}{k}$ for $k \in \{0,\ldots,\sigma\}$.
We rely on the following two claims, whose proofs are given below.
\begin{claim}
      For $0 \leq q < r_n$, if 
    $\|q \theta \| < \sigma \| r_n\theta \|$ then 
    $q \geq \frac{\varepsilon}{\sigma } r_n$.
\label{claim:dio}
\end{claim}
\begin{claim}
Let $n\in\mathbb N$ be sufficiently large that $\sigma \{r_n\theta\} <
1$.    Then  there exist $M\geq 0$ and $\varepsilon_1,
\varepsilon_2>0$ such that for all $m\in \mathbb N$,
  \begin{enumerate}
    \item if    $\{ m\theta +\alpha \}+\{\sigma r_n\theta \}<1$, then
    $w_{n,m}=0$;
\item if $m>Mr_n$ and $\{ m\theta +\alpha \} +\{\ell r_n\theta\} < 1< \{ m\theta +\alpha \} +\{(\ell+1) r_n\theta\}$ for some $\ell \in
\{0,\ldots,\sigma-1\}$, then $\varepsilon_1 m^{\sigma-2} \leq |w_{n,m}|
\leq \varepsilon_2 m^{\sigma-2}$.
\end{enumerate}
 \label{claim:long}
\end{claim}
For
$n \in \mathbb{N}$  let $0 \leq m_{n,1}< m_{n,2} <$ be an increasing
enumeration of $\Delta_n:=\{ m : w_{n,m}\neq 0\}$.
Note that by
Item~2 of Claim~\ref{claim:long} and
equidistribution of the sequence $\{m\theta+\alpha\}$, for all $n$
the set $\Delta_n$ is infinite.  Moreover, by
Item~1 of Claim~\ref{claim:long}, for
$j \in \{1,\ldots,\delta\}$ we have
$\|m_{n,j}\theta + \alpha \| < \sigma \|r_n\theta\|$ and
$\|m_{n,j+1}\theta + \alpha \| < \sigma \|r_n\theta\|$ and hence  $\|
(m_{n,j+1}-m_{n,j})\theta \| < \sigma \|r_n\theta\|$.  By
Claim~\ref{claim:dio} it follows that
$m_{n,j+1} - m_{n,j} \geq \varepsilon r_n\sigma^{-1}$, which establishes
the Expanding Gaps Property in Definition~\ref{def:stuttering}.

It remains to show the Polynomial-Variation Property: there exists $c_0\geq 0$ such
that for all $m'>m$ in $\Delta_n$ we have $|w_{n,m'}| \ll (m'-m)^{c_0}+|w_{n,m}|$.
We use Item~2 of Claim~\ref{claim:long}.  There are two cases.  First
suppose that $m>Mr_n$ for $M$ as in the claim.  Then
\[ |w_{n,m'}| \ll (m')^{\sigma-2} = (m'-m+m)^{\sigma -2} \ll
(m'-m)^{\sigma-2}+m^{\sigma-2} \ll (m'-m)^{\sigma-2}+|w_{n,m}| \, .\]
The second case is that $m \leq Mr_n$.  Since $m'-m \geq
\varepsilon\sigma^{-1}r_n$ we have $m'-m\geq
\frac{\varepsilon}{\varepsilon+\delta M}m'$.  From the fact that
$|w_{n,m'}| \leq (m+r_n)^{c_0}$ for some $c_0\geq 0$, it follows that 
$|w_{n,m'}| \ll (m'-m)^{c_0}$, which concludes the proof.
\end{proof}

  It remains to prove the two claims in the body of the proof of Theorem~\ref{thm:main1}.

\begin{proof}[Proof of Claim~\ref{claim:dio}]
  Assume that $\|q\theta \|<\sigma \|r_n\theta \|$.  Choose the
  largest $m$ such that $q_m\leq q$.
    By the law of best approximation we have
  $\|q_m\theta \|\leq \|q\theta\| < \sigma \|r_n
  \theta\|=\sigma\|q_{\ell_n}\theta \|$.  Then~\eqref{eq:FRAC} gives
\[ \frac{1}{(a_{m+1}+2)q_m} \leq \| q_m\theta \| \leq \sigma
  \|q_{\ell_n}\theta\| \leq \frac{\sigma}{a_{\ell_n+1}q_{\ell_n}} \]
  and we conclude  that
  $q_m \geq \frac{a_{{\ell_n}+1}q_{\ell_n}}{\sigma(a_{m+1}+2)}$.
  We also have $m<\ell_n$, since 
  $q_m \leq q<  r_n=q_{\ell_n}$; thus, by the defining property of $\ell_n$,  we have
  $\frac{a_{{\ell_n}+1}}{\sigma(a_{m+1}+2)} \geq
  \frac{\varepsilon}{\sigma}$. Combining the two previous bounds
  gives \[ q\geq
  q_m\geq \frac{a_{{\ell_n}+1}q_{\ell_n}}{\sigma(a_{m+1}+2)}
  \geq \frac{\varepsilon}{\sigma} q_{\ell_n}=
  \frac{\varepsilon}{\sigma} r_n \, .\]
This concludes the proof.
\end{proof}
  
\begin{proof}[Proof of Claim~\ref{claim:long}]
Given  $y \in \mathbb Z$, define the difference operator
    $\Delta_y : \mathbb Z[x]\rightarrow \mathbb Z[x]$ by
    $\Delta_y(g)(x) = g(x)-g(x+y)$.  Since $\Delta_y(g)$ has degree
    strictly less than that of $g$, we have that
    $(\Delta_y)^\sigma(f)(x) =\sum_{k=0}^{\sigma} b_k f(x+y k)$
is the zero polynomial in $\mathbb Z[x]$.  We now prove the two items
of the claim.

  1. If $\{ m\theta +\alpha \}+\{\sigma r_n\theta \}<1$, then
    \begin{eqnarray*}
      w_{n,m} &=&
\sum_{k=0}^{\sigma} b_k f( \lfloor (m+kr_n)\theta+\alpha \rfloor) \\
&=& \sum_{k=0}^{\sigma} b_k f( 
 \lfloor m\theta+\alpha \rfloor + \lfloor kr_n\theta \rfloor) \\
                                                  &=&
                                                      \sum_{k=0}^\sigma 
                                                      b_kf(\lfloor 
                                                      m\theta+\alpha 
                                                      \rfloor +k \lfloor 
                                                      r_n\theta\rfloor ) \\[2pt]
                                                     & = & 0 \, . 
    \end{eqnarray*}

    2.  If $\{ m\theta +\alpha \} +\{\ell r_n\theta\} < 1< \{ m\theta +\alpha \} +\{(\ell+1) r_n\theta\}$ for some $\ell \in
\{0,\ldots,\sigma-1\}$, then
 \begin{eqnarray*}
   w_{n,m} & =&  \sum_{k=0}^{\sigma} b_k f(\lfloor (m+kr_n)\theta 
                +\alpha \rfloor) \\
   &=& \sum_{k=0}^{\ell} b_k f(\lfloor m\theta 
       +\alpha \rfloor+k\lfloor r_n\theta\rfloor) 
+ \sum_{k=\ell+1}^\sigma b_k f(\lfloor m\theta 
       +\alpha \rfloor+1+k\lfloor r_n\theta\rfloor) \\
   &=& \sum_{k=0}^{\ell} b_k (f(\lfloor m\theta 
       +\alpha \rfloor+k\lfloor r_n\theta\rfloor) 
       - f(\lfloor m\theta 
       +\alpha \rfloor+1+k\lfloor r_n\theta\rfloor) \, .
 \end{eqnarray*}

 Define $p \in \mathbb Z[x,y]$ by $p(x,y):=\sum_{k=0}^{\ell}
 b_k(f(x+ky)-f(x+1+ky))$.  The equation above can be written
$w_{n,m}= p(\lfloor m\theta+\alpha \rfloor,\lfloor r_n\theta 
\rfloor)$.
 Since $f$ is not constant we have
 $\sigma\geq 2$.  By direct calculation,
 the coefficient of $x^{\sigma-2}$ in $p$ is 
the product of the leading coefficient of $f$ and 
\[ (1-\sigma)\sum_{k=0}^\ell b_k \, =\, (1-\sigma) \sum_{k=0}^{\ell}
  (-1)^k \binom{\sigma}{k} \,=\,  (-1)^{\ell} 
  (1-\sigma)\binom{\sigma-1}{\ell}\neq 0 \,  .  \] 
Thus for $M$  suitably large, there exists
$\varepsilon_1,\varepsilon_2>0$ such that
$x\geq M y$ implies $\varepsilon_0 x^{\sigma-2} \leq |p(x,y)|\leq
\varepsilon_2 x^{\sigma-2}$.  The claim follows.
\end{proof}

Combining Theorems~\ref{thm:main} and~\ref{thm:main1} we obtain our 
main result: 
\begin{theorem}
 Let $\theta,\alpha \in (0,1)$ with $\theta$ irrational and $\beta$ an
 algebraic number with $|\beta|>1$.  Given a 
 non-constant polynomial $f(x)\in \mathbb{Z}[x]$, the series 
 $\sum_{n=0}^\infty f(\lfloor n\theta+\alpha \rfloor)\beta^{-n}$ is transcendental. 
  \end{theorem}

  
\bibliography{bibliography}

\begin{thebibliography}{10}

\bibitem{AB1}
B.~Adamczewski and Y.~Bugeaud.
\newblock On the complexity of algebraic numbers {I}. expansions in integer
  bases.
\newblock {\em Annals of Mathematics}, 165:547--565, 2005.

\bibitem{BKLN21}
Y.~Bugeaud, D.~H. Kim, M.~Laurent, and A.~Nogueira.
\newblock On the {D}iophantine nature of the elements of {C}antor sets arising
  in the dynamics of contracted rotations.
\newblock {\em Annali Scuola Normale Superiore di Pisa - Classe Di Scienze},
  XXII:1681--1704, 2021.

\bibitem{FM}
S.~Ferenczi and C.~Mauduit.
\newblock Transcendence of numbers with a low complexity expansion.
\newblock {\em Journal of Number Theory}, 67(2):146--161, 1997.

\bibitem{KebisLOS024}
P.~Kebis, F.~Luca, J.~Ouaknine, A.~Scoones, and J.~Worrell.
\newblock On transcendence of numbers related to {S}turmian and
  {A}rnoux-{R}auzy words.
\newblock In {\em 51st International Colloquium on Automata, Languages, and
  Programming, {ICALP}}, volume 297 of {\em LIPIcs}, pages 144:1--144:15.
  Schloss Dagstuhl - Leibniz-Zentrum f{\"{u}}r Informatik, 2024.

\bibitem{BL}
M.~Laurent and Y.~Bugeaud.
\newblock Transcendence and continued fraction expansion of values of
  {H}ecke-{M}ahler series.
\newblock {\em Acta Arithmetica}, 209:59--90, 2023.

\bibitem{LEN97}
H.~W. Lenstra~Jr.
\newblock Finding small degree factors of lacunary polynomials.
\newblock {\em Number theory in progress}, 1:267--276, 1999.

\bibitem{LP}
J.~H. Loxton and A.~J. Van~der Poorten.
\newblock Arithmetic properties of certain functions in several variables
  {III}.
\newblock {\em Bulletin of the Australian Mathematical Society}, 16(1):15--47,
  1977.

\bibitem{LOW}
F.~Luca, J.~Ouaknine, and J.~Worrell.
\newblock On the transcendence of a series related to {S}turmian words, 2022.
\newblock To appear, Annali della Scuola Normale Superiore di Pisa.
\newblock \href {https://arxiv.org/abs/2204.08268} {\path{arXiv:2204.08268}}.

\bibitem{Masser1999}
D.~W. Masser.
\newblock Algebraic independence properties of the {H}ecke-{M}ahler series.
\newblock {\em Quarterly Journal of Mathematics}, 50:207--230, 1999.

\bibitem{Schlickewei76}
H.~P. Schlickewei.
\newblock Die p-adische verallgemeinerung des {S}atzes von
  {T}hue-{S}iegel-{R}oth-{S}chmidt.
\newblock {\em Journal für die reine und angewandte Mathematik},
  1976(288):86--105, 1976.

\end{thebibliography}
\end{document}